\documentclass[review]{elsarticle}
\usepackage{lineno,hyperref}
\usepackage{amsmath}
\usepackage{xcolor}
\usepackage{amssymb}
\usepackage{geometry}
\usepackage{amsthm}
\usepackage{graphicx}
\theoremstyle{plain}
\newtheorem{theorem}{Theorem}[section]
\newtheorem{lemma}[theorem]{Lemma}

\newtheorem{proposition}[theorem]{Proposition}
\newtheorem{Remark}[theorem]{Remark}

\modulolinenumbers[5]

\journal{Journal of \LaTeX\ Templates}

\begin{document}

\begin{frontmatter}

\title{On nth Level Fractional Derivatives: An Equivalent Representation and Applications to Inverse Problem}
\author[label1]{Asim Ilyas}\ead{ailyas1@uninsubria.it, asim.ilyas8753@gmail.com}
\author[label2]{Salman A. Malik}\ead{salman.amin.malik@gmail.com, salman\_amin@comsats.edu.pk}
\author[label2]{Kamran Suhaib}\ead{kamran.suhaib@gmail.com}
\address[label1]{Department of Science and High Technology, University of Insubria, Como, Italy}
\address[label2]{Department of Mathematics, COMSATS University Islamabad,\\ Park Road, Chak Shahzad Islamabad, Pakistan}

\begin{abstract}
This work contributes to the theory of nth level fractional derivative, where $n$ is a positive integer. An equivalent representation of 2nd level fractional derivative in terms of Riemann-Liouville fractional derivative is presented. We generalized our result and provide representation of nth level fractional derivative. As an application, we solve an inverse problem defined for a diffusion equation involving 2nd level fractional derivative.
  
\end{abstract}

\begin{keyword}
\texttt{Inverse problem, nth level fractional derivatives, Diffusion equation, Mittag-Leffler function}
\end{keyword}

\end{frontmatter}

\linenumbers
\section{Introduction and Preliminaries}
In the last few years, Fractional Calculus (FC) has earned a name in the field of mathematics, physics, engineering and many other fields \cite{kilbas,Hilfer,Samko-Kilbas}. Mathematical models involving factional derivative of different kinds described the accurate process in many physical sciences. Caputo and Riemann-Liouville are the commonly used fractional derivatives in mathematical modeling. In \cite{Hilfer}, Hilfer introduced a new definition of fractional derivative known as left and right sided fractional derivative of order $0<\rho<1$ and type $0\leq\nu\leq1$, usualy known as Hilfer fractional derivative (HFD). HFD is the generalization of the Riemann-Liouville fractional derivative (RLFD) and Caputo fractional derivative (CFD) with specific values of $0\leq\nu\leq1$. The properties as well as application related to HFD have been discussed in \cite{Hilfer,hilfer1}. In \cite{luchko}, Luchko presented a new generalization fractional ope rator of order $0<\rho\leq1$ and type $(\nu_{1},\nu_{2},...,\nu_{n})$ which is known as nth level fractional derivative (LFD). The nth LFD is the generalization of the $2$nd level fractional derivative (LFD) by plugging $n=2$, see \cite{luchko}. The $2$nd LFD is the generalization of into HFD, RLFD and CFD by taking the particular values of $(\nu_{1},\nu_{2})$. 

In this article, we established a relationship between the RLFD and the nth LFD by using some elementary arguments and also presented a relationship between RLFD and the $2$nd LFD. We discussed some basic properties of the $2$nd LFD and the nth LFD. 
\smallskip

The left and right sided Riemann-Liouville fractional integral of order $\rho$ for $f\in L^1([0,1])$ 
are \cite{Samko-Kilbas} defined as
$$J_{0_+,t}^{\rho}f(t)=\frac{1}{\Gamma(\rho)}
\int_{0}^{t}(t-\tau)^{\rho-1}f(\tau)
d\tau,$$
and
$$J_{1_-,t}^{\rho}f(t)=\frac{1}{\Gamma(\rho)}
\int_{t}^{1}(\tau-t)^{\rho-1}f(\tau)\;
d\tau,$$
respectively.

\begin{lemma}\label{semigrp}(cf. \cite{Samko-Kilbas}, formula 2.21)
	If $\rho_{1}>0,\;\rho_{2}>0$ and $f\in L^1([0,1]),$ then
	\begin{align*}
		\left(J_{0_+,t}^{\rho_{1}}J_{0_+,t}^{\rho_{2}}f\right)(t)=&\left(J_{0_+,t}^{\rho_{1}+\rho_{2}}f\right)(t),\qquad t\in[0,1]\;\; a.e.
	\end{align*}
\end{lemma}
The left and right sided RLFD of order $\rho\in(0,1)$ for $f\in L^{1}([0,1])$ are  defined \cite{Gara} as  \[
{}^{RL}D^{\rho}_{0_+,t}f(t)
=\frac{1}{\Gamma(1-\rho)}\frac{d}{dt}\int\limits^{t}_{0}
\frac{f(\tau)}{(t-\tau)^{\rho}}\;d\tau= \frac{d}{dt}J^{1-\rho}_{0_+,t}f(t),\qquad t >0,
\]
and
\[
{}^{RL}D^{\rho}_{1_-,t}f(t)
=\frac{1}{\Gamma(1-\rho)}\frac{d}{dt}
\int\limits^{1}_{t}
\frac{f(\tau)}{(\tau-t)^{\rho}}\;d\tau= \frac{d}{dt}J^{1-\rho}_{1_-,t}f(t),\qquad t <1,
\]
respectively.

\vspace{0.3cm}
Let us recall the first and second fundamental Theorem of Fractional Calculus for RLFD in the following proposition;
\begin{proposition}(\cite{kilbas}, Lemmas 2.4, 2.5(a))\label{22}
	Let $0<\rho<1$
	\begin{itemize}
			\item [(1)] If $f(t)\in L^{1}([0,1])$, then	$$\left({}^{RL}D_{0_+,t}^{\rho}J_{0_+,t}^{\rho}f\right)(t)=
			f(t).$$
			\item [(2)] If $f(t)\in J_{0_+,t}^{\rho}(L^{1}),$ then	$$\left(J_{0_+,t}^{\rho}{}^{RL}D_{0_+,t}^{\rho}f\right)(t)=
			f(t),$$
			where the spaces of functions $ J_{0_+,t}^{\rho}(L^{1})$ is defined as
			$$J_{0_+,t}^{\rho}(L^{1})=\{f:\;f=J_{0_+,t}^{\rho}\psi,\;\;\psi\in L^{1}([0,1])\}$$
	\end{itemize}
\end{proposition}
\smallskip

The RLFD of the function, i.e., $\Phi(t)=t^{\alpha},\; \alpha>-1$;
\begin{eqnarray}\label{deri}
	{}^{RL}D^{\rho}_{0_+,t}t^{\alpha}&=&\frac{\Gamma(\alpha+1)}{\Gamma(\alpha-\rho+1)}\;t^{\alpha-\rho},\quad 0<\rho<1.
\end{eqnarray}
\smallskip


The left and right sided CFD of $0<\rho<1$ for  $f\in AC([0,1])$ are defined \cite{Gara} as
\[
{}^{C}D^{\rho}_{0_+,t}f(t)
=\frac{1}{\Gamma(1-\rho)}\int\limits^{t}_{0}
\frac{f{'}(\tau)}{(t-\tau)^{\rho}}\;d\tau= J^{1-\rho}_{0_+,t} \frac{d}{dt}f(t),\qquad t >0,
\]and
\[
{}^{C}D^{\rho}_{1_-,t}f(t)
=\frac{1}{\Gamma(1-\rho)}\int\limits^{1}_{t}
\frac{f'(\tau)}{(\tau-t)^{\rho}}\;d\tau= J^{1-\rho}_{1_-,t} \frac{d}{dt}f(t),\qquad t <1,
\]
respectively.
\smallskip

The relation between the left and right sided RLFD and CFD of order $0<\rho<1$ is defined in \cite{kilbas} and is given by
\begin{align*}
	{}^{C}D^{\rho}_{0_+,t}f(t)=&{}^{RL}D^{\rho}_{0_+,t}\big(f(t)-f(0)\big),\\{}^{C}D^{\rho}_{1_-,t}f(t)=&{}^{RL}D^{\rho}_{1_-,t}\big(f(t)-f(1)\big).
\end{align*} 

The left and right sided HFD of order $\rho$ and type $\nu$ is \cite{Hilfer} defined as
\begin{align*}\Big({}^{H}D_{0_+,t}^{\rho, \nu}f\Big)(t)=& \left(J_{0_+,t}^{\nu(1-\rho)}\frac{d}{dt}J_{0_+,t}^{(1-\rho)(1-\nu)}f\right)(t),\\
	\Big({}^{H}D_{1_-,t}^{\rho, \nu}f\Big)(t)=& \left(J_{1_-,t}^{\nu(1-\rho)}\frac{d}{dt}J_{1_-,t}^{(1-\rho)(1-\nu)}f\right)(t),\end{align*} where $ 0<\rho<1$, $0\leq \nu\leq1,$  $f\in L^{1}([0,1])$,   ${J_{0_+,t}^{(1-\rho)(1-\nu)}}f\in AC([0,1]).$
\begin{Remark} We have the following comments, see \cite{Furati}\\
	The HFD ${}^{H}D_{0_+,t}^{\rho, \nu}f$ can be written as
	\begin{align*}\Big({}^{H}D_{0_+,t}^{\rho, \nu}f\Big)(t)&= \left(J_{0_+,t}^{\nu(1-\rho)}\frac{d}{dt}J_{0_+,t}^{1-\eta}f\right)(t)\\&=\left(J_{0_+,t}^{\nu(1-\rho)}{}^{RL}D_{0_+,t}^{\eta}f\right)(t)\\&=\left(J_{0_+,t}^{\eta-\rho}\;{}^{RL}D_{0_+,t}^{\eta}f\right)(t),\end{align*}
	where $\eta=\rho+\nu-\rho\nu$ and the parameter $\eta$ satisfies
	$$0<\eta\leq1,\qquad \eta\geq\rho,\qquad \eta>\nu,\qquad 1-\eta<1-\nu(1-\rho).$$ 
	The HFD is the generalization of the RLFD and the CFD.\\
	For $\nu =0,$ the HFD becomes the RLFD i.e.,
	$$D_{0_+,t}^{\rho,0}f(t)= \frac{d}{dt}J_{0_+,t}^{1-\rho}f(t):= {}^{RL}D_{0_+,t}^{\rho}f(t).$$
	For $\nu =1,$ the HFD reduces to the CFD i.e.,
	$$D_{0_+,t}^{\rho,1}f(t)= J_{0_+,t}^{1-\rho}\frac{d}{dt}f(t):= {}^{C}D_{0_+,t}^{\rho}f(t).$$
\end{Remark}
\vspace{0.3cm}
{
	For the better understanding of the reader, let us start with the necessary space of functions required to define the 2nd LFD. The space for the $2$nd LFD is given by
}
$$W_{2L}^{1}([0,1])=\Big\{f: J_{0_+,t}^{2-\rho-\nu_{1}-\nu_{2}}f\in AC([0,1]), J_{0_+,t}^{\nu_2}\frac{d}{dt}J_{0_+,t}^{2-\rho-\nu_{1}-\nu_2}f \in AC([0,1])\Big\}.$$ 
The $2$nd LFD of order $ \rho$ and type $(\nu_{1}, \nu_{2})$ for $f\in W_{2L}^{1}([0,1])$, is defined \cite{luchko} as
\begin{equation}\label{2lfd}
	{}^{2L}D_{0_+,t}^{\rho,(\nu_{1},\nu_2)}f(t)= \left(J_{0_+,t}^{\nu_{1}}\frac{d}{dt}J_{0_+,t}^{\nu_{2}}\frac{d}{dt}J_{0_+,t}^{(2-\rho-\nu_{1}-\nu_{2})}f\right)(t),
\end{equation}
$$0<\rho\leq1,\;\; 0\leq \nu_1,\;\; 0\leq\nu_2,\;\;\rho+\nu_{1}\leq1,\;\; \rho+\nu_{1}+\nu_{2}\leq2.$$
\begin{Remark}\label{rmk2nd} Let us provide some useful remarks before we proceed further;\\
	define $\xi_{1}=\rho+\nu_{1}+\nu_{2}-1,$ notice that under the condition on $\rho,\; \nu_{1},\;\nu_{2}$ is the definition of $2$nd LFD, we have $\; 0<\xi_{1}\leq1$ and we can write the $2$nd level fractional derivative ${}^{2L}D_{0_+,t}^{\rho,(\nu_{1},\nu_2)}$ as 
	\begin{align*}\Big({}^{2L}D_{0_+,t}^{\rho,(\nu_{1},\nu_2)}f\Big)(t)&= \left(J_{0_+,t}^{\nu_{1}}\frac{d}{dt}J_{0_+,t}^{\nu_{2}}\frac{d}{dt}J_{0_+,t}^{1-\xi_{1}}f\right)(t)\\&=\left(J_{0_+,t}^{\nu_{1}}\frac{d}{dt}J_{0_+,t}^{\xi_{1}-\rho-\nu_{1}+1}\;{}^{RL}D_{0_+,t}^{\xi_{1}}f\right)(t),\end{align*} \\
	The parameter $\xi_{1}$ satisfies
	$$0<\xi_{1}\leq1,\qquad \xi_{1}+1\geq\rho+\nu_{1},\qquad \xi_{1}+1\geq\rho+\nu_{2}.$$
	\newline
	\noindent and by taking $\xi_{2}=1-\nu_{2},\;\; (\nu_{2}=\xi_{1}-\rho-\nu_{1}+1),$\; we have \;$0<\xi_{2}\leq1$,  the $2$nd LFD ${}^{2L}D_{0_+,t}^{\rho,(\nu_{1},\nu_2)}$ can be written as
	\begin{eqnarray}\label{RL_2L}\Big({}^{2L}D_{0_+,t}^{\rho,(\nu_{1},\nu_2)}f\Big)(t)&=& \left(J_{0_+,t}^{\nu_{1}}\frac{d}{dt}J_{0_+,t}^{1-\xi_{2}}\;{}^{RL}D_{0_+,t}^{\xi_{1}}f\right)(t)\nonumber\\&=&\left(J_{0_+,t}^{\nu_{1}}\;{}^{RL}D_{0_+,t}^{\xi_{2}}\;{}^{RL}D_{0_+,t}^{\xi_{1}}f\right)(t),\end{eqnarray}
	The $2$nd LFD is the generalization of the HFD, the RLFD and the CFD, see \cite{Asim,luchko}.	\end{Remark}
\vspace{0.3cm}
In the similar fashion,  the space for the $3$rd LFD has the following form 
$$W_{3L}^{1}=\Big\{f: \Big(\displaystyle\prod_{k=j}^{3}\big(J_{0_+,t}^{\nu_{k}}\frac{d}{dt}\big)\Big) J_{0_+,t}^{(3-\rho-r_{3})}f \in AC([0,1]),\;j=2,3,3+1\Big\},$$
whereas $3$rd LFD of order $0<\rho\leq1$ is defined as
\begin{equation*}
	{}^{3L}D_{0_+,t}^{\rho,(\nu_{1},\nu_2, \nu_{3})}f(t)= \left(J_{0_+,t}^{\nu_{1}}\frac{d}{dt}J_{0_+,t}^{\nu_{2}}\frac{d}{dt}J_{0_+,t}^{\nu_{3}}\frac{d}{dt}J_{0_+,t}^{(2-\rho-\nu_{1}-\nu_{2}-\nu_{3})}f\right)(t),
\end{equation*}
$$ 0\leq \nu_1,\;\; 0\leq\nu_2,\;\; 0\leq\nu_3,\;\;\rho+\nu_{1}\leq1,\;\; \rho+\nu_{1}+\nu_{2}\leq2,\;\; \rho+\nu_{1}+\nu_{2}+\nu_{3}\leq3.$$
\begin{Remark} Let us mention some similar remarks about the $3$rd LFD;\\
	define $\xi_{1}=\rho+\nu_{1}+\nu_{2}+\nu_{3}-2,$ $\; 0<\xi_{1}\leq1$, we can write the $3$rd LFD ${}^{3L}D_{0_+,t}^{\rho,(\nu_{1},\nu_2, \nu_{3})};$ 
	\begin{align*}
		\Big({}^{3L}D_{0_+,t}^{\rho,(\nu_{1},\nu_2,\nu_{3})}f\Big)(t)=& \left(J_{0_+,t}^{\nu_{1}}\frac{d}{dt}J_{0_+,t}^{\nu_{2}}\frac{d}{dt}J_{0_+,t}^{\nu_{3}}\frac{d}{dt}J_{0_+,t}^{1-\xi_{1}}f\right)(t)\\=&\left(J_{0_+,t}^{\nu_{1}}\frac{d}{dt}J_{0_+,t}^{\xi_{1}-\rho-\nu_{1}-\nu_{3}+2}\frac{d}{dt}J_{0_+,t}^{\xi_{1}-\rho-\nu_{1}-\nu_{2}+2}\;{}^{RL}D_{0_+,t}^{\xi_{1}}f\right)(t),
	\end{align*}
	\newline
	\noindent Take $\xi_{2}=1-\nu_{3},$\;\;$(\nu_{3}=\xi_{1}-\rho-\nu_{1}-\nu_{2}+2),$ \;$0<\xi_{2}\leq1$,  the $3$rd LFD ${}^{3L}D_{0_+,t}^{\rho,(\nu_{1},\nu_2), \nu_{3}}$ can be written as
	\begin{align*}
		\Big({}^{3L}D_{0_+,t}^{\rho,(\nu_{1},\nu_2,\nu_{3})}f\Big)(t)=&\left(J_{0_+,t}^{\nu_{1}}\frac{d}{dt}J_{0_+,t}^{\xi_{1}-\rho-\nu_{1}-\nu_{3}+2}\frac{d}{dt}J_{0_+,t}^{1-\xi_{2}}\;{}^{RL}D_{0_+,t}^{\xi_{1}}f\right)(t)\\=&\left(J_{0_+,t}^{\nu_{1}}\frac{d}{dt}J_{0_+,t}^{\nu_{2}-\nu_{3}-\xi_{2}+1}\;{}^{RL}D_{0_+,t}^{\xi_{2}}\;{}^{RL}D_{0_+,t}^{\xi_{1}}f\right)(t).
	\end{align*}
	Let $\xi_{3}=1-\nu_{2},$\; \;$0<\xi_{3}\leq1$,  the $3$rd LFD ${}^{3L}D_{0_+,t}^{\rho,(\nu_{1},\nu_2), \nu_{3}}$ can be written as
	\begin{align*}
		\Big({}^{3L}D_{0_+,t}^{\rho,(\nu_{1},\nu_2,\nu_{3})}f\Big)(t)=&\left(J_{0_+,t}^{\nu_{1}}\frac{d}{dt}J_{0_+,t}^{1-\xi_{3}}\;{}^{RL}D_{0_+,t}^{\xi_{2}}\;{}^{RL}D_{0_+,t}^{\xi_{1}}f\right)(t)\\=&\left(J_{0_+,t}^{\nu_{1}}\;{}^{RL}D_{0_+,t}^{\xi_{3}}\;{}^{RL}D_{0_+,t}^{\xi_{2}}\;{}^{RL}D_{0_+,t}^{\xi_{1}}f\right)(t).
	\end{align*}
\end{Remark}
\vspace{0.3cm}
The space of functions for which the basic nth LFD  exists, is (see \cite{luchko} for more details)
\begin{equation}\label{nth}
	W_{nL}^{1}([0,1])=\Big\{f: \Big(\displaystyle\prod_{k=j}^{n}\big(J_{0_+,t}^{\nu_{k}}\frac{d}{dt}\big)\Big) J_{0_+,t}^{(n-\rho-r_{n})}f \in AC([0,1]),\;j=2,3,...,n+1\Big\},\end{equation}
where, for convenience, we define the following notation
\begin{equation}\label{rn est}
	r_{k}=\sum_{j=1}^{k}\nu_{j},\;\;k=1,2,...,n.
\end{equation}
The nth LFD of order $ 0<\rho\leq1$ and type $(\nu_{1},\nu_2,...,\nu_{n})$ for $f\in W_{nL}^{1}([0,1])$ is \cite{luchko} defined as

$$\Big({}^{nL}D_{0_+,t}^{\rho,(\nu_{1},\nu_2,...,\nu_{n})}f\Big)(t)=\bigg(\displaystyle\prod_{k=1}^{n}\big(J_{0_+,t}^{\nu_{k}}\frac{d}{dt}\big)J_{0_+,t}^{(n-\rho-r_{n})}f\bigg)(t),\;\; 0\leq r_k,\; \rho+r_{k}\leq k,$$
\begin{Remark}\label{rmknth} We have the following comments for the nth LFD;\\
	The nth LFD can be written as
	\begin{align*}
		\Big({}^{nL}D_{0_+,t}^{\rho,(\nu_{1},\nu_2,...,\nu_{n})}f\Big)(t)&=\bigg(\displaystyle\prod_{k=1}^{n}\big(J_{0_+,t}^{\nu_{k}}\frac{d}{dt}\big)J_{0_+,t}^{(n-\rho-r_{n})}f\bigg)(t)\\&= \left(J_{0_+,t}^{\nu_{1}}\displaystyle\prod_{k=1}^{n}\Big(\frac{d}{dt}J_{0_+,t}^{1-\xi_{k}}\Big)f\right)(t)\\&=\left(J_{0_+,t}^{\nu_{1}}\displaystyle\prod_{k=1}^{n}\;\big({}^{RL}D_{0_+,t}^{\xi_k}\big)f\right)(t),
	\end{align*}
		where $\xi_{1}=\rho+\displaystyle\sum_{k=1}^{n}\nu_{k}-(n-1)$,\; $0<\xi_{k}\leq1$ and $\xi_{i}=1-\nu_{n-i+2}$$,\;i=2,3,...,n$.\\
\end{Remark}
\noindent 


\vspace{0.3cm} 

After providing the fundamentals of fractional derivatives in the next section, we are going to present the main results of the article.  

\section{Alternate definition of the $2$nd level fractional derivative}
In this section, we are going to discuss the new way of expressing the 2nd LFD of order $\rho,\; 0<\rho<1$ and type $(\nu_{1},\nu_{2}),\;0\leq \nu_1,\; 0\leq\nu_2 .$ Before we proceed further, let us present the definition of the space $AC^{\xi}([0,1])$, see \cite{Bourdin}.
\newline

\nonumber Let $\xi_{k}\in(0,1),\;k=1,2,...n,$ we have  \begin{eqnarray}\label{define-f}
	AC^{\xi_{k}}([0,1])=\Big\{f:[0,1]\to\mathbb{R}^{n}/ f(t)=\frac{C_k}{\Gamma(\xi_{k})}t^{\xi_{k}-1}+\big(J_{0_+,t}^{\xi_{k}}\psi\big)(t), \;\;\text{ for some } C_k\in\mathbb{R}^n \text{ and } \psi\in L^1([0,1])\Big\},\end{eqnarray} 

Let us recall the following result (Theorem 2.1 of \cite{Bourdin}) 
\begin{theorem}\label{rei-theo}\cite{Bourdin}
	Let $\xi\in(0,1)$ and $f\in L^{1}([0,1])$. The function $f$ has the RLFD of order $\xi$ iff $f\in AC^{\xi}([0,1])$. Then
	\begin{equation*}\left({}^{RL}D_{0_+,t}^{\xi}f\right)(t)=\psi(t),\qquad 
		\mbox{and} \qquad \left(J_{0_+,t}^{1-\xi}f\right)(t)|_{t=0}=C_{0},\end{equation*}
	where $C_{0}\in\mathbb{R}^{n}$ is a constant and $\psi\in L^{1}([0,1])$.
\end{theorem}
We prove a similar result to Theorem \ref{rei-theo} for the $2$nd LFD.
\begin{theorem}\label{2nd-con}
	Let $0<\rho\leq1$,\; $0\leq \nu_1,\; 0\leq\nu_2,\;\rho+\nu_{1}\leq1,\; \rho+\nu_{1}+\nu_{2}\leq2$,\;$f, f_{1}\in L^{1}([0,1])$, $\xi_{1}=\rho+\nu_{1}+\nu_{2}-1,$ and $\xi_{2}=1-\nu_{2}$. The function $f$ has the $2$nd LFD of order $\rho$ and types $(\nu_1, \nu_2)$ iff $f\in AC^{\xi_{1}}([0,1])\; \& \;f_{1}\in AC^{\xi_{2}}([0,1]),$
	we have
	$$\left({}^{2L}D_{0_+,t}^{\rho,(\nu_{1},\nu_{1})}f\right)(t)=\left(J_{0_+,t}^{\nu_{1}}\psi\right)(t),$$
	and
	\begin{eqnarray}\label{cdn 4}
		\left.
		\begin{aligned}
			\left({}^{RL}D_{0_+,t}^{\xi_{1}}f\right)(t)=&f_{1}(t),\qquad\left(J_{0_+,t}^{1-\xi_{1}}f\right)(t)|_{t=0}=C_1,\\
			\left({}^{RL}D_{0_+,t}^{\xi_{2}}f_{1}\right)(t)=&\psi(t),\qquad\left(J_{0_+,t}^{1-\xi_{2}}f_{1}\right)(t)|_{t=0}=C_2.
		\end{aligned}
		\right\}
	\end{eqnarray}
	where $C_{1}$ and $C_{2}$ are the constants appeared in the definitions of 	$f$ and $f_{1}:=\frac{d}{dt}J_{0_+,t}^{1-\xi_{1}}f$, respectively and $\psi\in L^{1}([0,1])$.
	Furthermore, we have the following relation

	\begin{eqnarray}=&&
	\left({}^{2L}D_{0_+,t}^{\rho,(\nu_{1},\nu_{1})}f\right)(t)=	{}^{RL}D_{0_+}^{\rho}\Bigg(f(.)-{C_3} t^{\rho+\nu_{1}+\nu_{2}-2}-{C_4} t^{\rho+\nu_{1}-1}\Bigg)(t).\label{ass}
	\end{eqnarray}
	where 
	\begin{eqnarray}\label{cdn 34}
		\left.
		\begin{aligned}
			\frac{\big(J_{0_+,t}^{1-\xi_{1}}f\big)(t)|_{t=0}}{\Gamma(\rho+\nu_{1}+\nu_{2}-1)}=&C_3,\\
			\frac{\Big(J_{0_+,t}^{\nu_{2}}\frac{d}{dt}J_{0_+,t}^{1-\xi_{1}}f\Big)(t)|_{t=0}}{\Gamma(\rho+\nu_{1})}=&C_4.
		\end{aligned}
		\right\}
	\end{eqnarray}
\end{theorem}
\begin{proof}
	Assume that $f$ possesses the $2$nd LFD of order $\rho$ and types $(\nu_1, \nu_2)$, then $J_{0_+,t}^{1-\xi_{1}}f\in AC([0,1]),\;  J_{0_+,t}^{1-\xi_{2}}f_{1}\in AC([0,1]),$ the RLFD of order $\xi_{1}$ and $\xi_{2}$ of the functions $f$ and $f_{1}$ exist. From definition of $2$nd LFD \eqref{2lfd}, Equation \eqref{RL_2L} and conditions \eqref{cdn 4}, we have
	\begin{align*}
		\left({}^{2L}D_{0_+,t}^{\rho,(\nu_{1},\nu_{2})}f\right)(t)=&\left(J_{0_+,t}^{\nu_{1}}\;{}^{RL}D_{0_+,t}^{\xi_{2}}\;{}^{RL}D_{0_+,t}^{\xi_{1}}f\right)(t)\nonumber\\=&\left(J_{0_+,t}^{\nu_{1}}{}^{RL}D_{0_+,t}^{\xi_{2}}f_{1}\right)(t)\nonumber\\=&\left(J_{0_+,t}^{\nu_{1}}\psi\right)(t).\end{align*}
	Conversely suppose that $f\in AC^{\xi_{1}}([0,1])$ and $f_{1}\in AC^{\xi_{2}}([0,1])$ and conditions \eqref{cdn 4}, we have
	\begin{align*}
		\left({}^{2L}D_{0_+,t}^{\rho,(\nu_{1},\nu_{2})}f\right)(t)=&\left(J_{0_+,t}^{\nu_{1}}\psi\right)(t).\end{align*}
	Using Theorem \ref{rei-theo} and conditions \eqref{cdn 4} , we have
	\begin{align*}
		\left({}^{2L}D_{0_+,t}^{\rho,(\nu_{1},\nu_{2})}f\right)(t)=&\left(J_{0_+,t}^{\nu_{1}}{}^{RL}D_{0_+,t}^{\xi_{2}}f_{1}\right)(t)\\=&\left(J_{0_+,t}^{\nu_{1}}\;{}^{RL}D_{0_+,t}^{\xi_{2}}\;{}^{RL}D_{0_+,t}^{\xi_{1}}f\right)(t).\end{align*}
	This shows that there exists the $2$nd LFD of order $\rho$ and type $(\nu_1, \nu_2)$ of $f$.
	\smallskip
	\vspace{0.2cm}
	
	Let us provide the identity given in \eqref{ass}. Since,
	$J_{0_+,t}^{1-\xi_{1}}f\in AC([0,1]),\;J_{0_+,t}^{1-\xi_{2}}f_{1}\in AC([0,1]).$  In view of Remark \ref{rmk2nd} and Lemma \ref{semigrp}, we have
	\begin{align*}
		\Big({}^{2L}D_{0_+,t}^{\rho,(\nu_{1},\nu_2)}f\Big)(t)&= \left(J_{0_+,t}^{\nu_{1}}\frac{d}{dt}J_{0_+,t}^{\nu_{2}}\frac{d}{dt}J_{0_+,t}^{1-\xi_{1}}f\right)(t)\\&
		=	\left(\frac{d}{dt}J_{0_+,t}^{1}J_{0_+,t}^{\nu_{1}}\frac{d}{dt}J_{0_+,t}^{\nu_{2}}\frac{d}{dt}J_{0_+,t}^{1-\xi_{1}}f\right)(t)
		\\&
		=	\left(\frac{d}{dt}J_{0_+,t}^{\nu_{1}}J_{0_+,t}^{1}\frac{d}{dt}J_{0_+,t}^{\nu_{2}}\frac{d}{dt}J_{0_+,t}^{1-\xi_{1}}f\right)(t).
	\end{align*}
	Using fundamental theorem of Calculus, we have
	\begin{align*}
		\Big({}^{2L}D_{0_+,t}^{\rho,(\nu_{1},\nu_2)}f\Big)(t)=&	\left(\frac{d}{dt}J_{0_+,t}^{\nu_{1}}\Big(J_{0_+,t}^{\nu_{2}}\frac{d}{dt}J_{0_+,t}^{1-\xi_{1}}f-J_{0_+,t}^{\nu_{2}}\frac{d}{dt}J_{0_+,t}^{1-\xi_{1}}f(t)|_{t=0}\Big)\right)(t)\\=&	\frac{d}{dt}\left(J_{0_+,t}^{\nu_{1}}J_{0_+,t}^{\nu_{2}}\frac{d}{dt}J_{0_+,t}^{1-\xi_{1}}f\right)(t)\\&-\frac{d}{dt}\left(J_{0_+,t}^{\nu_{1}}\Big(J_{0_+,t}^{\nu_{2}}\frac{d}{dt}J_{0_+,t}^{1-\xi_{1}}f(t)|_{t=0}\Big)\right)(t)\\=&	\frac{d}{dt}\left(\frac{d}{dt}J_{0_+,t}^{\nu_{1}+\nu_{2}}J_{0_+,t}^{1}\frac{d}{dt}J_{0_+,t}^{1-\xi_{1}}f\right)(t)\\&-\frac{d}{dt}\Bigg(\frac{\Big(J_{0_+,t}^{\nu_{2}}\frac{d}{dt}J_{0_+,t}^{1-\xi_{1}}f\Big)(t)|_{t=0}}{\Gamma(\nu_{1}+1)}\;t^{\nu_{1}}\Bigg).
	\end{align*}
	Again using fundamental theorem of Calculus, we have
	\begin{align*}
		\Big({}^{2L}D_{0_+,t}^{\rho,(\nu_{1},\nu_2)}f\Big)(t)=&	\frac{d}{dt}\left(\frac{d}{dt}J_{0_+,t}^{\nu_{1}+\nu_{2}}\Big(J_{0_+,t}^{1-\xi_{1}}f-J_{0_+,t}^{1-\xi_{1}}f(t)|_{t=0}\Big)\right)(t)\\&-\frac{d}{dt}\Bigg(\frac{\Big(J_{0_+,t}^{\nu_{2}}\frac{d}{dt}J_{0_+,t}^{1-\xi_{1}}f\Big)(t)|_{t=0}}{\Gamma(\nu_{1}+1)}t^{\nu_{1}}\Bigg)\\=&
		\frac{d}{dt}\left(\frac{d}{dt}J_{0_+,t}^{\nu_{1}+\nu_{2}}J_{0_+,t}^{1-\xi_{1}}f\right)(t)-\frac{d}{dt}\left(\frac{d}{dt}J_{0_+,t}^{\nu_{1}+\nu_{2}}\big(J_{0_+,t}^{1-\xi_{1}}f(t)|_{t=0}\big)\right)(t)\\&-\frac{d}{dt}\Bigg(\frac{\Big(J_{0_+,t}^{\nu_{2}}\frac{d}{dt}J_{0_+,t}^{1-\xi_{1}}f\Big)(t)|_{t=0}}{\Gamma(\nu_{1}+1)}\;t^{\nu_{1}}\Bigg)\\=&
		\frac{d}{dt}\left(J_{0_+,t}^{1-\rho}f\right)(t)-\frac{d}{dt}\Bigg(\frac{\big(J_{0_+,t}^{1-\xi_{1}}f\big)(t)|_{t=0}}{\Gamma(\nu_{1}+\nu_{2}+1)}\;t^{\nu_{1}+\nu_{2}}\Bigg)\\&-\frac{d}{dt}\Bigg(\frac{\Big(J_{0_+,t}^{\nu_{2}}\frac{d}{dt}J_{0_+,t}^{1-\xi_{1}}f\Big)(t)|_{t=0}}{\Gamma(\nu_{1}+1)}\;t^{\nu_{1}}\Bigg),
	\end{align*}
	which implies
	\begin{align*}	
		\Big({}^{2L}D_{0_+,t}^{\rho,(\nu_{1},\nu_2)}f\Big)(t)=&
		{}^{RL}D_{0_+}^{\rho}\Bigg(f(.)-\frac{\big(J_{0_+,t}^{1-\xi_{1}}f\big)(t)|_{t=0}}{\Gamma(\rho+\nu_{1}+\nu_{2}-1)}\;t^{\rho+\nu_{1}+\nu_{2}-2}\nonumber\\&-\frac{\Big(J_{0_+,t}^{\nu_{2}}\frac{d}{dt}J_{0_+,t}^{1-\xi_{1}}f\Big)(t)|_{t=0}}{\Gamma(\rho+\nu_{1})}\;t^{\rho+\nu_{1}-1}\Bigg)(t).	
	\end{align*}
	The proof of this theorem is complete.
\end{proof}
\noindent\textbf{Fundamental Theorems of FC for the 2nd LFD:} The first and second Fundamental Theorems of FC for the 2nd LFD are discussed in \cite{luchko}. The second fundamental theorem of 2nd LFD can obtain by applying RLFI in formula \ref{ass}  
	\begin{align*}	
		\Big(J_{0_+,t}^{\rho}{}^{2L}D_{0_+,t}^{\rho,(\nu_{1},\nu_2)}f\Big)(t)=&
		\Bigg(f(.)-\frac{\big(J_{0_+,t}^{1-\xi_{1}}f\big)(t)|_{t=0}}{\Gamma(\rho+\nu_{1}+\nu_{2}-1)}\;t^{\rho+\nu_{1}+\nu_{2}-2}\nonumber\\&-\frac{\Big(J_{0_+,t}^{\nu_{2}}\frac{d}{dt}J_{0_+,t}^{1-\xi_{1}}f\Big)(t)|_{t=0}}{\Gamma(\rho+\nu_{1})}\;t^{\rho+\nu_{1}-1}\Bigg)(t).	
\end{align*}
\begin{Remark} We have the following remarks;\\
	$\bullet$ For $\nu_{1}=\nu_{1}(1-\rho)$ and $\nu_{2}=1$, the $2$nd LFD becomes the HFD, i.e., 
	\begin{align*}{}^{2L}D_{0_+,t}^{\rho,(\nu_{1}(1-\rho),1)}f(t)=&
		{}^{RL}D_{0_+}^{\rho}\Bigg(f(.)-\frac{\big(J_{0_+,t}^{2-\rho-\nu_{1}(1-\rho)-1}f\big)(t)|_{t=0}}{\Gamma(\rho+\nu_{1}(1-\rho)+1-1)}\;t^{\rho+\nu_{1}(1-\rho)+1-2}\\&-\frac{\Big(J_{0_+,t}^{1}\frac{d}{dt}J_{0_+,t}^{2-\rho-\nu_{1}(1-\rho)-1}f\Big)(t)|_{t=0}}{\Gamma(\rho+\nu_{1}(1-\rho))}\;t^{\rho+\nu_{1}(1-\rho)-1}\Bigg)(t)\\
		=&
		{}^{RL}D_{0_+}^{\rho}\Bigg(f(.)-\frac{\big(J_{0_+,t}^{1-\rho-\nu_{1}(1-\rho)}f\big)(t)|_{t=0}}{\Gamma(\rho+\nu_{1}(1-\rho))}\;t^{\rho+\nu_{1}(1-\rho)-1}\\&-\frac{\Big(J_{0_+,t}^{1-\rho-\nu_{1}(1-\rho)}f\Big)(t)|_{t=0}}{\Gamma(\rho+\nu_{1}(1-\rho))}\;t^{\rho+\nu_{1}(1-\rho)-1}\Bigg)(t).
	\end{align*}
	Notice that if the one of the conditions, $\nu_{2}<1,\; 1<\rho+\nu_{1}+\nu_{2}$ (page 13, see \cite{luchko}) does not hold, then kernel of the $2$nd LFD becomes one-dimensional and one of the above conditions is equal to zero. Hence, we obtain the following result  \cite{Kamocki} 
	\begin{align*}
		{}^{2L}D_{0_+,t}^{\rho,(\nu_{1}(1-\rho),1)}f(t)=&	{}^{RL}D_{0_+}^{\rho}\Bigg(f(.)-\frac{\big(J_{0_+,t}^{1-\rho-\nu_{1}(1-\rho)}f\big)(t)|_{t=0}}{\Gamma(\rho+\nu_{1}(1-\rho))}\;t^{\rho+\nu_{1}(1-\rho)-1}\Bigg)(t)\\=&:{}^{H}D_{0_+,t}^{\rho,\nu_{1}}f(t).
	\end{align*}
	$\bullet$ For $\nu_{1} =0$ and $\nu_{2}=1$, the $2$nd LFD becomes the RLFD from \eqref{ass}, we have
	\begin{align*}{}^{2L}D_{0_+,t}^{\rho,(0,1)}f(t)=&
		{}^{RL}D_{0_+}^{\rho}\Bigg(f(.)-\frac{\big(J_{0_+,t}^{2-\rho-0-1}f\big)(t)|_{t=0}}{\Gamma(\rho+0+1-1)}\;t^{\rho+0+1-2}\\&-\frac{\Big(J_{0_+,t}^{1}\frac{d}{dt}J_{0_+,t}^{2-\rho-0-1}f\Big)(t)|_{t=0}}{\Gamma(\rho+0)}\;t^{\rho+0-1}\Bigg)(t)\\
		=&
		{}^{RL}D_{0_+}^{\rho}\Bigg(f(.)-\frac{\big(J_{0_+,t}^{1-\rho}f\big)(t)|_{t=0}}{\Gamma(\rho)}\;t^{\rho-1}-\frac{\Big(J_{0_+,t}^{1-\rho}f\Big)(t)|_{t=0}}{\Gamma(\rho)}\;t^{\rho-1}\Bigg)(t),
	\end{align*}
	Since ${}^{RL}D_{0_+}^{\rho}t^{\rho-1}=0$, then we have
	\begin{align*}{}^{2L}D_{0_+,t}^{\rho,(0,1)}f(t)=:&
		{}^{RL}D_{a_+,t}^{\rho}f(.)(t).
	\end{align*}
	$\bullet$ For $\nu_{1} =1-\rho$ and $\nu_{2}=0$, the $2$nd LFD reduces to the CFD, i.e.,
	\begin{align*}{}^{2L}D_{0_+,t}^{\rho,(1-\rho,0)}f(t)=&
		{}^{RL}D_{0_+}^{\rho}\Bigg(f(.)-\frac{\big(J_{0_+,t}^{2-\rho-(1-\rho)-0}f\big)(t)|_{t=0}}{\Gamma(\rho+(1-\rho)+0-1)}\;t^{\rho(1-\rho)+0-2}\\&-\frac{\Big(J_{0_+,t}^{0}\frac{d}{dt}J_{0_+,t}^{2-\rho-(1-\rho)-1}f\Big)(t)|_{t=0}}{\Gamma(\rho+(1-\rho))}\;t^{\rho+(1-\rho)-1}\Bigg)(t)\\=&
		{}^{RL}D_{0_+}^{\rho}\Bigg(f(.)-\frac{\big(J_{0_+,t}^{1}f\big)(t)|_{t=0}}{\Gamma(1)}-\frac{\Big(\frac{d}{dt}J_{0_+,t}^{1}f\Big)(t)|_{t=0}}{\Gamma(1)}\Bigg)(t),	\end{align*}
	Since, the term $\frac{\big(J_{0_+,t}^{1}f\big)(t)|_{t=0}}{\Gamma(1)},$ is equal to zero, we have
	\begin{align*}
		{}^{2L}D_{0_+,t}^{\rho,(1-\rho,0)}f(t)=&
		{}^{RL}D_{0_+}^{\rho}\Big(f(.)-f(t)|_{t=0}\Big)(t)=: {}^{C}D_{a_+,t}^{\rho}f(t).
	\end{align*}
\end{Remark}
\begin{Remark}
	The basic $2$nd LFD becomes the RLFD for the functions belonging to t the space \cite{luchko} $$W_{2L}^{0}([0,1])=\left\{f\in W^0:\frac{d}{dt}J_{0_+,t}^{\nu_1}f=J_{0_+,t}^{\nu_{1}}\frac{df}{dt}\;\; \mbox{and}\;\;\frac{d}{dt}J_{0_+,t}^{\nu_{1}+\nu_{2}}f=J_{0_+,t}^{\nu_{1}+\nu_{2}}\frac{df}{dt} \right\},$$
\end{Remark}
where $W^0=J_{0_+,t}^{\nu}(L^{1}(0,1)).$
\section{Alternate definition of the nth level fractional derivative}
In this section, we are going to present the new way of expressing the nth LFD of order $\rho,\;0<\rho\leq1$ and type $(\nu_{1},\nu_{2},...,\nu_{n}).$ We will prove that it can be defined as a generalized RLFD of order $\rho$.
\begin{theorem}\label{nthfunc}
	Let $0<\rho\leq1$, $0\leq\nu_{k}$, $\rho+r_{k}\leq k$, where $r_{k}$ is given in \eqref{rn est}, $\xi_{1}=\rho+\displaystyle\sum_{k=1}^{n}\nu_{k}-(n-1)$,\;  $\xi_{i}=1-\nu_{n-i+2}$$,\;i=2,3,...,n$ and $f\in L^{1}([0,1],\mathbb{R}^{n})$. The function $f$ has the nth  LFD of order $\rho$ and type $(\nu_1, \nu_2,...,\nu_n)$  iff $f\in AC^{\xi_{1}}([0,1])$\; \& \; $f_{k}\in AC^{\xi_{k+1}}([0,1])\;\; k=1,2,..,n$, we have
	$$\Big({}^{nL}D_{0_+,t}^{\rho,(\nu_{1},\nu_2,...,\nu_{n})}f\Big)(t)=\left(J_{0_+,t}^{\nu_1}\psi\right)(t),$$
	where
	\begin{eqnarray}\label{cdn 5}
		\left.
		\begin{aligned}
			\left({}^{RL}D_{0_+,t}^{\xi_{1}}f\right)(t)=&f_{1}(t),\quad\left(J_{0_+,t}^{1-\xi_{1}}f\right)(t)|_{t=0}=C_1,\\
			\left({}^{RL}D_{0_+,t}^{\xi_{k}}f_{k-1}\right)(t)=&f_{k}(t),\;\left(J_{0_+,t}^{1-\xi_{k}}f_{k-1}\right)(t)|_{t=0}=C_k,\; k=2,..,n-1,\\
			\left({}^{RL}D_{0_+,t}^{\xi_{n}}f_{n-1}\right)(t)=&\psi(t),\quad
			\left(J_{0_+,t}^{1-\xi_{n}}f_{n}\right)(t)|_{t=0}=C_{n}.
		\end{aligned}
		\right\}
	\end{eqnarray}
	Furthermore, we have
	\begin{eqnarray}
	\Big({}^{nL}D_{0_+,t}^{\rho,(\nu_{1},\nu_2,...,\nu_{n})}f\Big)(t)	=	{}^{RL}D_{0_+}^{\rho}\Bigg(f(.)-\sum_{k=1}^{n}\frac{\Big(\displaystyle\prod_{j=k+2}^{n}\Big(\frac{d}{dt}J_{0_+,t}^{\nu_{j}}\Big)\frac{d}{dt}J_{0_+,t}^{1-\xi_{1}}f\Big)(t)|_{t=0}}{\Gamma(\rho+r_{k}-(k-1))}\;t^{\rho+r_{k}-k}\Bigg)(t).\label{nthf}
	\end{eqnarray}
\end{theorem}
\begin{proof}
	Let us assume that $f$ has the nth LFD of order $\rho$ and types $(\nu_1, \nu_2,...,\nu_{n})$, then $J_{0_+,t}^{1-\xi_{1}}f\in AC([0,1])$ \&
	$J_{0_+,t}^{1-\xi_{k+1}}f_{k}\in AC([0,1]),\;\; k=1,2,..,n$  so there exists RLFDs of order $\xi_{k}$ of the functions $f$\; \&\; $f_{k-1}, \;\; k=1,2,..,n.$ Consequently, we have
	\begin{align}
		\Big({}^{nL}D_{0_+,t}^{\rho,(\nu_{1},\nu_2,...,\nu_{n})}f\Big)(t)=&\left(J_{0_+,t}^{\nu_{1}}\displaystyle\prod_{k=1}^{n}\;\big({}^{RL}D_{0_+,t}^{\xi_k}\big)f\right)(t)\\=&\left(J_{0_+,t}^{\nu_{1}}\displaystyle\prod_{k=2}^{n}\;\big({}^{RL}D_{0_+,t}^{\xi_k}\big){}^{RL}D_{0_+,t}^{\xi_1}f\right)(t)\\=&\left(J_{0_+,t}^{\nu_{1}}{}^{RL}D_{0_+,t}^{\xi_n}\displaystyle\prod_{k=2}^{n-1}\;\big({}^{RL}D_{0_+,t}^{\xi_k}\big)f_{k-1}\right)(t)\\=&\left(J_{0_+,t}^{\nu_{1}}{}^{RL}D_{0_+,t}^{\xi_n}f_{n-1}\right)(t)\\=&\left(J_{0_+,t}^{\nu_1}\psi\right)(t).\end{align}
	Conversely, for $f\in AC^{\xi_{1}}([0,1])$ \& $f_{k}\in AC^{\xi_{k+1}}([0,1])\;\; k=1,2,..,n$, we have
	$$\Big({}^{nL}D_{0_+,t}^{\rho,(\nu_{1},\nu_2,...,\nu_{n})}f\Big)(t)=\left(J_{0_+,t}^{\nu_1}\phi\right)(t).$$
	Due to Theorem \ref{rei-theo} and conditions \eqref{cdn 5}, we have
	\begin{align*}
		\left({}^{nL}D_{0_+,t}^{\rho,(\nu_{1},\nu_2,...,\nu_{n})}f\right)(t)=&\left(J_{0_+,t}^{\nu_{1}}{}^{RL}D_{0_+,t}^{\xi_{n}}f_{n-1}\right)(t)\\=&\left(J_{0_+,t}^{\nu_{1}}\;{}^{RL}D_{0_+,t}^{\xi_{n}}\displaystyle\prod_{k=2}^{n-1}\big(\;{}^{RL}D_{0_+,t}^{\xi_{k}}\big)f_{k-1}\right)(t)\\=&\left(J_{0_+,t}^{\nu_{1}}\;{}^{RL}D_{0_+,t}^{\xi_{n}}\displaystyle\prod_{k=2}^{n}\big(\;{}^{RL}D_{0_+,t}^{\xi_{k}}\big){}^{RL}D_{0_+,t}^{\xi_1}f\right)(t)\\=&\left(J_{0_+,t}^{\nu_{1}}\;\displaystyle\prod_{k=1}^{n}\big(\;{}^{RL}D_{0_+,t}^{\xi_{k}}\big)f\right)(t),\end{align*}
	which implies that there exists the nth LFD of order $\rho$ and type $(\nu_1, \nu_2,...,\nu_{n})$ of $f$.
	\smallskip 
	
	\noindent
	The nth LFD exists if and only if  $f\in AC^{\xi_{1}}([0,1])$ \& $f_{k}\in AC^{\xi_{k+1}}([0,1])\;\; k=1,2,..,n$.
	Since, $J_{0_+,t}^{1-\xi_{1}}f\in AC([0,1])$ \&
	$J_{0_+,t}^{1-\xi_{k+1}}f_{k}\in AC([0,1]),$  then by virtue of Remark \ref{rmknth} and Lemma \ref{semigrp}, we have 
	\begin{align*}
		\Big({}^{nL}D_{0_+,t}^{\rho,(\nu_{1},\nu_2,...,\nu_{n})}f\Big)(t)&=\left(J_{0_+,t}^{\nu_{1}}\displaystyle\prod_{k=1}^{n}\;\big({}^{RL}D_{0_+,t}^{\xi_k}\big)f\right)(t)\\&=\left(J_{0_+,t}^{\nu_{1}}\displaystyle\prod_{k=2}^{n}\Big(\frac{d}{dt}J_{0_+,t}^{1-\xi_{k}}\Big)\frac{d}{dt}J_{0_+,t}^{1-\xi_{1}}f\right)(t)\\&= \left(J_{0_+,t}^{\nu_{1}}\displaystyle\prod_{k=2}^{n}\Big(\frac{d}{dt}J_{0_+,t}^{\nu_{k}}\Big)\frac{d}{dt}J_{0_+,t}^{1-\xi_{1}}f\right)(t)\\&
		=	\left(\frac{d}{dt}J_{0_+,t}^{1}J_{0_+,t}^{\nu_{1}}\displaystyle\prod_{k=2}^{n}\Big(\frac{d}{dt}J_{0_+,t}^{\nu_{j}}\Big)\frac{d}{dt}J_{0_+,t}^{1-\xi_{1}}f\right)(t)
		\\&
		=	\left(\frac{d}{dt}J_{0_+,t}^{\nu_{1}}J_{0_+,t}^{1}\frac{d}{dt}J_{0_+,t}^{\nu_{2}}\displaystyle\prod_{k=3}^{n}\Big(\frac{d}{dt}J_{0_+,t}^{\nu_{k}}\Big)\frac{d}{dt}J_{0_+,t}^{1-\xi_{1}}f\right)(t).
	\end{align*}
	Using fundamental theorem of Calculus, we have
	\begin{align*}
		\Big({}^{nL}D_{0_+,t}^{\rho,(\nu_{1},\nu_2,...,\nu_{n})}f\Big)(t)=&	\Bigg[\frac{d}{dt}J_{0_+,t}^{\nu_{1}}\Bigg(J_{0_+,t}^{\nu_{2}}\displaystyle\prod_{k=3}^{n}\Big(\frac{d}{dt}J_{0_+,t}^{\nu_{k}}\Big)\frac{d}{dt}J_{0_+,t}^{1-\xi_{1}}f\\&-J_{0_+,t}^{\nu_{2}}\displaystyle\prod_{k=3}^{n-1}\Big(\frac{d}{dt}J_{0_+,t}^{\nu_{k}}\Big)\frac{d}{dt}J_{0_+,t}^{1-\xi_{1}}f(t)|_{t=0}\Bigg)\Bigg](t)\\=&	\frac{d}{dt}\left(J_{0_+,t}^{\nu_{1}}J_{0_+,t}^{\nu_{2}}\displaystyle\prod_{k=3}^{n}\Big(\frac{d}{dt}J_{0_+,t}^{\nu_{k}}\Big)\frac{d}{dt}J_{0_+,t}^{1-\xi_{1}}f\right)(t)\\&-\frac{d}{dt}J_{0_+,t}^{\nu_{1}}\left(J_{0_+,t}^{\nu_{2}}\displaystyle\prod_{k=3}^{n}\Big(\frac{d}{dt}J_{0_+,t}^{\nu_{k}}\Big)\frac{d}{dt}J_{0_+,t}^{1-\xi_{1}}f(t)|_{t=0}\right)(t).\end{align*} This implies that 
	\begin{align*}\Big({}^{nL}D_{0_+,t}^{\rho,(\nu_{1},\nu_2,...,\nu_{n})}f\Big)(t)=&	\frac{d}{dt}\left(\frac{d}{dt}J_{0_+,t}^{1}J_{0_+,t}^{\nu_{1}+\nu_{2}}\frac{d}{dt}J_{0_+,t}^{\nu_{3}}\displaystyle\prod_{k=4}^{n}\Big(\frac{d}{dt}J_{0_+,t}^{\nu_{k}}\Big)\frac{d}{dt}J_{0_+,t}^{1-\xi_{1}}f\right)(t)\\&-\frac{d}{dt}\Bigg(\frac{\Big(J_{0_+,t}^{\nu_{2}}\displaystyle\prod_{k=3}^{n}\Big(\frac{d}{dt}J_{0_+,t}^{\nu_{k}}\Big)\frac{d}{dt}J_{0_+,t}^{1-\xi_{1}}f\Big)(t)|_{t=0}}{\Gamma(\nu_{1}+1)}\;t^{\nu_{1}}\Bigg)\\=&	\frac{d}{dt}\left(\frac{d}{dt}J_{0_+,t}^{\nu_{1}+\nu_{2}}J_{0_+,t}^{1}\frac{d}{dt}J_{0_+,t}^{\nu_{3}}\displaystyle\prod_{k=4}^{n}\Big(\frac{d}{dt}J_{0_+,t}^{\nu_{k}}\Big)\frac{d}{dt}J_{0_+,t}^{1-\xi_{1}}f\right)(t)\\&-\frac{d}{dt}\Bigg(\frac{\Big(J_{0_+,t}^{\nu_{2}}\displaystyle\prod_{k=4}^{n}\Big(\frac{d}{dt}J_{0_+,t}^{\nu_{k}}\Big)\frac{d}{dt}J_{0_+,t}^{1-\xi_{1}}f\Big)(t)|_{t=0}}{\Gamma(\nu_{1}+1)}\;t^{\nu_{1}}\Bigg).
	\end{align*}
	Again fundamental theorem of Calculus, we have
	\begin{align*}
		\Big({}^{nL}D_{0_+,t}^{\rho,(\nu_{1},\nu_2,...,\nu_{n})}f\Big)(t)=&\frac{d}{dt}\Bigg[\frac{d}{dt}J_{0_+,t}^{\nu_{1}+\nu_{2}}\bigg(J_{0_+,t}^{\nu_{3}}\displaystyle\prod_{k=4}^{n}\Big(\frac{d}{dt}J_{0_+,t}^{\nu_{k}}\Big)\frac{d}{dt}J_{0_+,t}^{1-\xi_{1}}f\\&-J_{0_+,t}^{\nu_{3}}\displaystyle\prod_{k=4}^{n}\Big(\frac{d}{dt}J_{0_+,t}^{\nu_{k}}\Big)\frac{d}{dt}J_{0_+,t}^{1-\xi_{1}}f(t)|_{t=0}\bigg)\Bigg](t)\\&-\frac{d}{dt}\Bigg(\frac{\Big(J_{0_+,t}^{\nu_{2}}\displaystyle\prod_{k=3}^{n}\Big(\frac{d}{dt}J_{0_+,t}^{\nu_{k}}\Big)\frac{d}{dt}J_{0_+,t}^{1-\xi_{1}}f\Big)(t)|_{t=0}}{\Gamma(\nu_{1}+1)}\;t^{\nu_{1}}\Bigg),\\=&\frac{d}{dt}\left(\frac{d}{dt}J_{0_+,t}^{\nu_{1}+\nu_{2}}J_{0_+,t}^{\nu_{3}}\displaystyle\prod_{k=4}^{n}\Big(\frac{d}{dt}J_{0_+,t}^{\nu_{k}}\Big)\frac{d}{dt}J_{0_+,t}^{1-\xi_{1}}f\right)(t)\\&-\frac{d}{dt}\Bigg[\frac{d}{dt}J_{0_+,t}^{\nu_{1}+\nu_{2}}\left(J_{0_+,t}^{\nu_{3}}\displaystyle\prod_{k=4}^{n}\Big(\frac{d}{dt}J_{0_+,t}^{\nu_{k}}\Big)\frac{d}{dt}J_{0_+,t}^{1-\xi_{1}}f(t)|_{t=0}\right)\Bigg](t)\\&-\frac{d}{dt}\Bigg(\frac{\Big(J_{0_+,t}^{\nu_{2}}\displaystyle\prod_{k=3}^{n}\Big(\frac{d}{dt}J_{0_+,t}^{\nu_{k}}\Big)\frac{d}{dt}J_{0_+,t}^{1-\xi_{1}}f\Big)(t)|_{t=0}}{\Gamma(\nu_{1}+1)}\;t^{\nu_{1}}\Bigg),\end{align*}which implies
	\begin{align*}
		\Big({}^{nL}D_{0_+,t}^{\rho,(\nu_{1},\nu_2,...,\nu_{n})}f\Big)(t)=&\frac{d}{dt}\left(\frac{d^2}{dt^{2}}J_{0_+,t}^{1}J_{0_+,t}^{\nu_{1}+\nu_{2}+\nu_{3}}\frac{d}{dt}J_{0_+,t}^{\nu_{4}}\displaystyle\prod_{k=5}^{n}\Big(\frac{d}{dt}J_{0_+,t}^{\nu_{k}}\Big)\frac{d}{dt}J_{0_+,t}^{1-\xi_{1}}f\right)(t)\\&-\frac{d}{dt}\Bigg(\frac{\Big(J_{0_+,t}^{\nu_{3}}\displaystyle\prod_{k=4}^{n}\Big(\frac{d}{dt}J_{0_+,t}^{\nu_{k}}\Big)\frac{d}{dt}J_{0_+,t}^{1-\xi_{1}}f\Big)(t)|_{t=0}}{\Gamma(\nu_{1}+\nu_{2}+1)}\;t^{\nu_{1}+\nu_{2}}\Bigg)	\\&-\frac{d}{dt}\Bigg(\frac{\Big(J_{0_+,t}^{\nu_{2}}\displaystyle\prod_{k=3}^{n}\Big(\frac{d}{dt}J_{0_+,t}^{\nu_{k}}\Big)\frac{d}{dt}J_{0_+,t}^{1-\xi_{1}}f\Big)(t)|_{t=0}}{\Gamma(\nu_{1}+1)}\;t^{\nu_{1}}\Bigg),
	\end{align*}
	\begin{align*}
		\Big({}^{nL}D_{0_+,t}^{\rho,(\nu_{1},\nu_2,...,\nu_{n})}f\Big)(t)=&\frac{d}{dt}\left(\frac{d^2}{dt^{2}}J_{0_+,t}^{\nu_{1}+\nu_{2}+\nu_{3}}J_{0_+,t}^{1}\frac{d}{dt}J_{0_+,t}^{\nu_{4}}\displaystyle\prod_{k=5}^{n}\Big(\frac{d}{dt}J_{0_+,t}^{\nu_{k}}\Big)\frac{d}{dt}J_{0_+,t}^{1-\xi_{1}}f\right)(t)\\&-\frac{d}{dt}\Bigg(\frac{\Big(J_{0_+,t}^{\nu_{3}}\displaystyle\prod_{k=4}^{n}\Big(\frac{d}{dt}J_{0_+,t}^{\nu_{k}}\Big)\frac{d}{dt}J_{0_+,t}^{1-\xi_{1}}f\Big)(t)|_{t=0}}{\Gamma(\nu_{1}+\nu_{2}+1)}\;t^{\nu_{1}+\nu_{2}}\Bigg)\\&-\frac{d}{dt}\Bigg(\frac{\Big(J_{0_+,t}^{\nu_{2}}\displaystyle\prod_{k=3}^{n}\Big(\frac{d}{dt}J_{0_+,t}^{\nu_{k}}\Big)\frac{d}{dt}J_{0_+,t}^{1-\xi_{1}}f\Big)(t)|_{t=0}}{\Gamma(\nu_{1}+1)}\;t^{\nu_{1}}\Bigg).	\end{align*}
	Similarly, using term by term fundamental theorem of Calculus, we can get
\begin{align*}
		\left({}^{nL}D_{0_+,t}^{\rho,(\nu_{1},\nu_{2},...,\nu_{n})}f\right)(t)=&
		{}^{RL}D_{0_+}^{\rho}\Bigg(f(.)-\sum_{k=1}^{n}\frac{J_{0_+,t}^{\nu_{k+1}}\Big(\displaystyle\prod_{j=k+2}^{n}\Big(\frac{d}{dt}J_{0_+,t}^{\nu_{j}}\Big)\frac{d}{dt}J_{0_+,t}^{1-\xi_{1}}f\Big)(t)|_{t=0}}{\Gamma(\rho+r_{k}-(k-1))}\;t^{\rho+r_{k}-k}\Bigg)(t).
	\end{align*}
	The proof of this theorem is complete.
\end{proof}
\noindent\textbf{Fundamental Theorems of FC for the nth LFD:} The first and second Fundamental Theorems of FC for the nth LFD are discussed in \cite{luchko1}. The second fundamental theorem of nth LFD can obtain by applying RLFI in formula \ref{nthf}  
\begin{align*}
		\left(J_{0_+,t}^{\rho}{}^{nL}D_{0_+,t}^{\rho,(\nu_{1},\nu_{2},...,\nu_{n})}f\right)(t)=&
		\Bigg(f(.)-\sum_{k=1}^{n}\frac{J_{0_+,t}^{\nu_{k+1}}\Big(\displaystyle\prod_{j=k+2}^{n}\Big(\frac{d}{dt}J_{0_+,t}^{\nu_{j}}\Big)\frac{d}{dt}J_{0_+,t}^{1-\xi_{1}}f\Big)(t)|_{t=0}}{\Gamma(\rho+r_{k}-(k-1))}\;t^{\rho+r_{k}-k}\Bigg)(t).
\end{align*}
\begin{Remark}
	For $n=2$, the nth LFD reduces to $2$nd LFD given by \eqref{ass} and the result becomes to Theorem \ref{2nd-con}. 
\end{Remark}
\begin{Remark}
	The basic nth LFD interpolate with the RLFD in the space \cite{luchko} $$W_{nL}^{0}([0,1])=\left\{f\in W^0:\frac{d}{dt}J_{0_+,t}^{r_k}f=J_{0_+,t}^{r_k}\frac{df}{dt},\;k=1,2,3,...,n\right\}.$$
\end{Remark}
\section{Applications of $2$nd Level Fractional Derivative}
In this section, we are going to discuss some application of $2$nd LFD, given by \eqref{ass}.
\begin{Remark}
	The new representation of $2$nd LFD given  by \eqref{ass} can be written as
	$$\left({}^{2L}D_{0_+,t}^{\rho,(\nu_{1},\nu_{2})}f\right)(t)$$\begin{eqnarray}\label{asim}
		&=&
		{}^{RL}D_{0_+}^{\rho}f(t)-\frac{\big(J_{0_+,t}^{1-\xi_{1}}f\big)(t)|_{t=0}}{\Gamma(\nu_{1}+\nu_{2}-1)} t^{\nu_{1}+\nu_{2}-2}-\frac{\Big(J_{0_+,t}^{\nu_{2}}\frac{d}{dt}J_{0_+,t}^{1-\xi_{1}}f\Big)(t)|_{t=0}}{\Gamma(\nu_{1})} t^{\nu_{1}-1}.	
	\end{eqnarray}
	For determining the $2$nd LFD of any function $f$, it is sufficient to calculate RLFD ${}^{RL}D_{0_+}^{\rho}f(t)$ and $J_{0_+,t}^{1-\rho}f(t)|_{t=0}$ by using the \eqref{asim}. Hence, it is easier to calculate $2$nd LFD by using \eqref{asim} instead of using  \eqref{2lfd}. 
\end{Remark}  
\begin{proposition}\cite{Samko-Kilbas}\label{12}
	Assume $f\in L^1([0,1])$. then we have
	\begin{equation*}
		f\in J_{0_+,t}^{\rho}(L^{1})\quad \Rightarrow\quad J_{0_+,t}^{1-\rho}f \in AC([0,1]), \quad \mbox{and} \quad J_{0_+,t}^{1-\rho}f(t)|_{t=0}=0.
	\end{equation*}
\end{proposition}
\begin{Remark}
	In \cite{luchko}, the Laplace transform of the $2$nd LFD \eqref{2lfd} is given by 
	\begin{eqnarray*}
		\mathcal{L}\Big\{\left({}^{2L}D_{0_+,t}^{\rho,(\nu_{1},\nu_{2})}f\right)(t);s\Big\}:= s^{\rho}\mathcal{L}\{f(t);s\}-s^{-\nu_{1}-\nu_{2}+1}\big(J_{0_+,t}^{1-\xi_{1}}f\big)(t)|_{t=0}-s^{-\nu_{1}}\Big(J_{0_+,t}^{\nu_{2}}\frac{d}{dt}J_{0_+,t}^{1-\xi_{1}}f\Big)(t)|_{t=0} .
	\end{eqnarray*}
	By using an alternate definition given in \eqref{2lfd}, we have the expression of Laplace transform 
	\begin{align*}	
		\mathcal{L}&\Bigg\{{}^{RL}D_{0_+}^{\rho}\Bigg(f(.)-\frac{\big(J_{0_+,t}^{1-\xi_{1}}f\big)(t)|_{t=0}}{\Gamma(\rho+\nu_{1}+\nu_{2}-1)}\;t^{\rho+\nu_{1}+\nu_{2}-2}-\frac{\Big(J_{0_+,t}^{\nu_{2}}\frac{d}{dt}J_{0_+,t}^{1-\xi_{1}}f\Big)(t)|_{t=0}}{\Gamma(\rho+\nu_{1})}\;t^{\rho+\nu_{1}-1}\Bigg)(t);s\Bigg\}\\=&s^{\rho}\mathcal{L}\{f(t);s\}-\big(J_{0_+,t}^{1-\rho}f\big)(t)|_{t=0}-s^{-\nu_{1}-\nu_{2}+1}\big(J_{0_+,t}^{1-\xi_{1}}f\big)(t)|_{t=0}-s^{-\nu_{1}}\Big(J_{0_+,t}^{\nu_{2}}\frac{d}{dt}J_{0_+,t}^{1-\xi_{1}}f\Big)(t)|_{t=0}  .	
	\end{align*}

	The new definition of $2$nd LFD given by  \eqref{ass} can be obtained by using the Laplace transform of the $2$nd LFD. Indeed, we have $\xi_{1},\;\xi_{2}>\rho$ for $AC^{\xi_{1}}([0,1]), AC^{\xi_2}([0,1]) \subset J_{0_+,t}^{\rho}(L^{1}([0,1]))$, see (\cite{id}, Theorem 28). Then we have $J_{0_+,t}^{1-\rho}f(t)|_{t=0}=0,$ so
	$$\mathcal{L}\{{}^{RL}D_{0_+}^{\rho}f(t); s\}= s^{\rho}\mathcal{L}\{f(t);s\}.$$
	
\end{Remark}
\begin{lemma}\label{11}
	If \;$0<\rho\leq1$,\; $0\leq \nu_1,\; 0\leq\nu_2,\;\rho+\nu_{1}\leq1,\; \rho+\nu_{1}+\nu_{2}\leq2$,\;  $f\in J_{0_+,t}^{\xi_{1}}(L^{1})$ with  $\xi_{1}=\rho+\nu_{1}+\nu_{2}-1$ and $\;  J_{0_+,t}^{1-\xi_{2}}f_{1}\in J_{0_+,t}^{\xi_{2}}(L^{1})$ with $\xi_{2}=1-\nu_{2}$, then we have
	\begin{equation*}
		\left({}^{2L}D_{0_+,t}^{\rho,(\nu_{1},\nu_{2})}f\right)(t)=
		{}^{RL}D_{0_+}^{\rho}f(t).	
	\end{equation*}
\end{lemma}
\begin{proof}
	From Equation \eqref{asim}, we have
	\begin{align*}
		\left({}^{2L}D_{0_+,t}^{\rho,(\nu_{1},\nu_{2})}f\right)(t)=&
		{}^{RL}D_{0_+}^{\rho}f(t)-\frac{\big(J_{0_+,t}^{1-\xi_{1}}f\big)(t)|_{t=0}}{\Gamma(\nu_{1}+\nu_{2}-1)} t^{\nu_{1}+\nu_{2}-2}-\frac{\Big(J_{0_+,t}^{1-\xi_{2}}f_{1}\Big)(t)|_{t=0}}{\Gamma(\nu_{1})} t^{\nu_{1}-1},	
	\end{align*}
	where $f_{1}(t):=\frac{d}{dt}J_{0_+,t}^{1-\xi_{1}}f(t)$.
	
	\noindent By using the Proposition \ref{12}, we get the relation
	\begin{equation*}
		\left({}^{2L}D_{0_+,t}^{\rho,(\nu_{1},\nu_{2})}f\right)(t)=
		{}^{RL}D_{0_+}^{\rho}f(t).	
	\end{equation*}
\end{proof}
\begin{lemma}
	If \;$0<\rho\leq1$,\; $0\leq \nu_1,\; 0\leq\nu_2,\;\rho+\nu_{1}\leq1,\; \rho+\nu_{1}+\nu_{2}\leq2$,\;  $f\in J_{0_+,t}^{\xi_{1}}(L^{1})$ with  $\xi_{1}=\rho+\nu_{1}+\nu_{2}-1$ and $\;  f_{1}\in J_{0_+,t}^{\xi_{2}}(L^{1})$ with $\xi_{2}=1-\nu_{2}$. Then we have
	\begin{eqnarray}
		\left(J_{0_+,t}^{\rho}{}^{2L}D_{0_+,t}^{\rho,(\nu_{1},\nu_{2})}f\right)(t)&=&
		f(t),\label{1}\\	
		\left({}^{2L}D_{0_+,t}^{\rho,(\nu_{1},\nu_{2})}J_{0_+,t}^{\rho}f\right)(t)&=&
		f(t).\label{2}
	\end{eqnarray}
\end{lemma}
\begin{proof}
	From Proposition \ref{12}, we have
	\begin{align*}
		f\in J_{0_+,t}^{\xi_{1}}(L^{1})&, \qquad \Rightarrow \qquad J_{0_+,t}^{1-\xi_{1}}f\in AC([0,1]),\; \\
		f_{1}\in J_{0_+,t}^{\xi_{2}}(L^{1})&,\qquad \Rightarrow \qquad J_{0_+,t}^{1-\xi_{2}}f_{1}\in AC([0,1]).
	\end{align*}
	By virtue of Equation \eqref{asim}, we have
	\begin{align*}
		\left({}^{2L}D_{0_+,t}^{\rho,(\nu_{1},\nu_{2})}f\right)(t)=&
		{}^{RL}D_{0_+}^{\rho}f(t)-\frac{\big(J_{0_+,t}^{1-\xi_{1}}f\big)(t)|_{t=0}}{\Gamma(\nu_{1}+\nu_{2}-1)} t^{\nu_{1}+\nu_{2}-2}-\frac{\Big(J_{0_+,t}^{1-\xi_{2}}f_{1}\Big)(t)|_{t=0}}{\Gamma(\nu_{1})} t^{\nu_{1}-1}. 	
	\end{align*}
	Due to Lemma \ref{11}, we obtain
	\begin{equation*}
		\left({}^{2L}D_{0_+,t}^{\rho,(\nu_{1},\nu_{2})}f\right)(t)=
		{}^{RL}D_{0_+}^{\rho}f(t).	
	\end{equation*}
	Applying RLFI on the left-hand side, we have
	\begin{equation*}
		\left(J_{0_+,t}^{\rho}{}^{2L}D_{0_+,t}^{\rho,(\nu_{1},\nu_{2})}f\right)(t)=
		J_{0_+,t}^{\rho}{}^{RL}D_{0_+}^{\rho}f(t).	
	\end{equation*}
	Using Proposition \ref{22}, we get the equality of Equation \eqref{1}
	\begin{align*}
		\left(J_{0_+,t}^{\rho}{}^{2L}D_{0_+,t}^{\rho,(\nu_{1},\nu_{1})}f\right)(t)=&
		f(t).
	\end{align*}
	Similarly, we can prove the equality of Equation \eqref{2}.
	\begin{align*}
		\left({}^{2L}D_{0_+,t}^{\rho,(\nu_{1},\nu_{2})}J_{0_+,t}^{\rho}f\right)(t)=&
		f(t).
	\end{align*}
\end{proof}
\subsection{Diffusion Equation}
In this section, We will discuss an inverse problem. 
\vspace{0.3cm}

Consider the following initial boundary value problem 
\begin{eqnarray}\label{peoblem}
	\left.
	\begin{aligned}
		\Big({}^{2L}D_{0_+,t}^{\rho,(\nu_{1},\nu_2)}u\Big)(x,t)&=u_{xx}(x,t)+f(x),\qquad (x,t)\in (0,1)\times(0,T),\\
		u(1,t)=0,\quad  &u_{x}(0,t)= u_{x}(1,t),\qquad\qquad x\in [0,T],\\
		J_{0_+,t}^{1-\xi_{2}}\frac{d}{dt}J_{0_+,t}^{1-\xi_{1}}&u(x,t)|_{t=0} = \phi(x), \qquad\qquad x \in [0,1], \\
		J_{0_+,t}^{1-\xi_{1}}&u(x,t)|_{t=0} = \psi(x),\qquad\qquad x \in [0,1]. 
	\end{aligned}
	\right\}
\end{eqnarray}
where ${}^{2L}D_{0_+,t}^{\rho,(\nu_{1},\nu_2)}$ stands for the 2nd LFD of order $\rho\in(0,1)$ and type $(\nu_{1},\nu_{2})$, that is given by Eq. \eqref{2lfd}.
\smallskip

\noindent
The inverse problem is to recover a pair of function $\{u(x,t), f(x)\}$ for the given system \eqref{peoblem}. For finding the source term $f(x)$, we need the following additional data
\begin{eqnarray}\label{overdeterminationcnd}
	u(x,T)&=&\varphi(x).
\end{eqnarray}
Using formula \eqref{ass}, problem \eqref{peoblem} can be written in the following form
\begin{eqnarray}\label{peoblemRL}
	\left.
	\begin{aligned}
		\Big({}^{RL}D_{0_+,t}^{\rho} u\Big)(x,t)&=u_{xx}(x,t)+f(x),\qquad (x,t)\in (0,1)\times(0,T),\\
		u(1,t)=0,\quad  &u_{x}(0,t)= u_{x}(1,t),\qquad\qquad x\in [0,T],\\J_{0_+,t}^{1-\rho}u(x,t)|_{t=0}&=0,\qquad\qquad\qquad\qquad x \in [0,1],\\
		J_{0_+,t}^{1-\xi_{2}}\frac{d}{dt}J_{0_+,t}^{1-\xi_{1}}&u(x,t)|_{t=0} = \phi(x), \qquad\qquad\qquad x \in [0,1], \\
		J_{0_+,t}^{1-\xi_{1}}&u(x,t)|_{t=0} = \psi(x),\qquad
		\qquad\qquad x \in [0,1]. 
	\end{aligned}
	\right\}
\end{eqnarray}
Now, we will obtain the solution of the given system \eqref{peoblemRL} through an bi-orthogonal expansion obtained by using the method of
variable separation. The spectral problem corresponding to \eqref{peoblem}
\begin{eqnarray}\label{sproblem}
	X^{''}(x)&=&-\lambda X(x), \qquad X(1)=0,\qquad
	X^{'}(0)=X^{'}(1),
\end{eqnarray}
where $\lambda$ is the spectral parameter.
\smallskip

\noindent The eigenvalues of the system \eqref{sproblem} are $\lambda_{k}=2\pi k, \;\; k\in\mathbb{N}$.
The set of eigenfunctions of the spectral problem are given by
\begin{eqnarray}\label{eigenfunction}
	\Big\{X_{0}(x)&=& 2(1-x),\; X_{1k}(x)=4(1-x)\cos(\lambda_{k}x),\; X_{2k}(x)=4\sin (\lambda_{k}x)\Big\}.
\end{eqnarray}
The set of eigenfunctions in Eq. \eqref{eigenfunction} form a Riesz basis for the space $L^{2}((0,1))$ but not orthogonal, see \cite{Ionkin}. So, the conjugate problem is given by
\begin{eqnarray}\label{sproblem1}
	Y^{''}(x)&=&-\lambda Y(x),\qquad Y^{'}(1)=0,\qquad
	Y(0)=Y(1),
\end{eqnarray}  
The set of eigenfunctions of Eq. \eqref{sproblem1} are given by
\begin{eqnarray}\label{eigenfunction1}
	\Big\{Y_{0}(x)= 1, \quad Y_{1k}(x)=\cos(\lambda_{k}x),\quad Y_{2k}(x)=x\sin (\lambda_{k}x))\Big\},
\end{eqnarray}
\smallskip

\noindent  The sets of eigenfunctions given by Eqs. \eqref{eigenfunction} and \eqref{eigenfunction1} form a bi-orthogonal system of functions, see \cite{moise}.
\newline
\noindent \textbf{Series Representation Solution:} The solution of the inverse problem can be expressed as
\begin{eqnarray}
	u(x, t) &=& U_0(t)X_0(x) + \sum_{k=1}^{\infty}
	U_{1k}(t)X_{1k}(x) + \sum_{k=1}^{\infty}U_{2k}(t)X_{2k}(x),\label{vsolution}\\
	f(x)&=&f_0 X_0(x) + \sum_{k=1}^{\infty}
	f_{1k}X_{1k}(x) + \sum_{k=1}^{\infty}f_{2k}X_{2k}(x),\label{fsolution}
\end{eqnarray}
where $U_0(t)$, $U_{1k}(t)$, $U_{2k}(t)$, $f_0$, $f_{1k}$ \& $f_{2k}$ are unknowns. By orthogonality of eigenfunctions and due to formula \eqref{ass} and Eq. \eqref{peoblemRL}, we get

\begin{align}
U_{0}(t)=&\frac{\phi_{0}t^{\rho+\nu_{1}-1}}{\Gamma(\rho+\nu_{1})}+\frac{\psi_{0}t^{\rho+\nu_{1}+\nu_{2}-2}}{\Gamma(\rho+\nu_{1}+\nu_{2}-1)}+\frac{f_{0}t^{\rho}}{\Gamma(\rho+1)},\label{T1}\\
U_{1k}(t)=&{\phi_{1k}}t^{\rho+\nu_{1}-1}{E}_{{\rho},\rho+\nu_{1}}(-\lambda^{2}_{k}t^{\rho})+{\psi_{1k}}t^{\rho+\nu_{1}+\nu_{2}-2}{E}_{{\rho},\rho+\nu_{1}+\nu_{2}-1}(-\lambda^{2}_{k}t^{\rho})+ f_{1k}t^{\rho}{E}_{{\rho, \rho+1}}(-\lambda^{2}_{k}t^{\rho}),\label{T2}\\
U_{2k}(t)=& {\phi_{2k}}t^{\rho+\nu_{1}-1}{E}_{{\rho},\rho+\nu_{1}}(-\lambda^{2}_{k}t^{\rho})+{\psi_{2k}}t^{\rho+\nu_{1}+\nu_{2}-1}l{E}_{{\rho},\rho+\nu_{1}+\nu_{2}-1}(-\lambda^{2}_{k}t^{\rho})+2 \lambda_{k}U_{1k}(t)\ast t^{\rho-1}{E}_{{\rho, \rho}}(-\lambda^{2}_{k}t^{\rho})\nonumber\\&+f_{2k} t^{\rho}{E}_{{\rho, \rho+1}}(-\lambda^{2}_{k}t^{\rho}),\label{T3}
\end{align}	
where ${E}_{{\rho, \nu}}(-\lambda_{k}t^{\rho})$ is the two parameter Mittag-Leffler function \cite{Gorenflo-Kilbas-Minardi} which is defined as
$$E_{\rho,\nu}(-\lambda_{k}t^{\rho})=\sum_{k=0}^{\infty}\frac{(-\lambda_{k}t^{\rho})^k}
{\Gamma(\rho k+\nu)},\hspace{1cm} Re(\rho)>0,\qquad\nu \in \mathbb{C},$$
and $$f_{j}=\int_{0}^{1}{f(x)Y_{j}(x)}ds,\qquad j=0,1k, 2k, \;\;k\in\mathbb{N}.$$ 

\smallskip

\noindent By virtue of Eq. \eqref{overdeterminationcnd}, the unknowns of $f_{0}, f_{1k}$ and $ f_{2k}$  yield to have the following expressions
\begin{align}
f_{0}=&\frac{\Gamma(\rho+1)}{T^{\rho}}\bigg(\varphi_{0}-\frac{\phi_{0}t^{\rho+\nu_{1}-1}}{\Gamma(\rho+\nu_{1})}-\frac{\psi_{0}t^{\rho+\nu_{1}+\nu_{2}-2}}{\Gamma(\rho+\nu_{1}+\nu_{2}-1)}\bigg),\label{f0}\\
f_{1k}=&\frac{1}{\mathcal{E}_{{\rho},\rho+1}(-\lambda^{2}_{k}T^{\rho})}\bigg(\varphi_{1k}-{\phi_{1k}}t^{\rho+\nu_{1}-1}{E}_{{\rho},\rho+\nu_{1}}(-\lambda^{2}_{k}t^{\rho})\nonumber\\&-{\psi_{1k}}t^{\rho+\nu_{1}+\nu_{2}-2}{E}_{{\rho},\rho+\nu_{1}+\nu_{2}-1}(-\lambda^{2}_{k}t^{\rho})\bigg),\label{f1}\\
f_{2k}=&\frac{1}{\mathcal{E}_{{\rho}, \rho+1}(-\lambda^{2}_{k}T^{\rho})}\bigg(\varphi_{2k}-{\phi_{2k}}t^{\rho+\nu_{1}-1}{E}_{{\rho},\rho+\nu_{1}}(-\lambda^{2}_{k}t^{\rho})\nonumber\\&-{\psi_{2k}}t^{\rho+\nu_{1}+\nu_{2}-2}{E}_{{\rho},\rho+\nu_{1}+\nu_{2}-1}(-\lambda^{2}_{k}t^{\rho})+2 \lambda_{k}U_{1k}(t)\ast t^{\rho-1}{E}_{{\rho, \rho}}(-\lambda^{2}_{k}t^{\rho})\bigg),\label{f2}
\end{align}
where $$\varphi_{j}=\int_{0}^{1}\varphi(x) Y_{j}(x)dx,\;\qquad j=0, 1k, 2k,\;\; k\in\mathbb{N}.$$
\smallskip

\noindent The similar solution, i.e., $\{u(x,t), f(x)\}$ can be obtained from the Problem \eqref{peoblem}. Let us mention the solution of Problem \eqref{peoblem} is discussed in \cite{Asim}.

Notice that the inverse problem is to determine the source term and temperature concentration for a diffusion equation with nth LFD is considered in \cite{Malik} and we can find the inverse problem using formula \eqref{nthf}  which is studied in \cite{Malik}.

\section{Conclusions} \label{sec:6}

The nth LFD proposed by Y. Luchko \cite{luchko} of order $\rho$ and $n$ parameters $\nu_{1},\nu_{2},...,\nu_{n}$ is proved to have an equivalent representation in term of RLFD under certain conditions stated in Theorem \ref{nthfunc}. From our results when $n=2$, the 2nd LFD is the generalization of the well-known HFD, CFD, and RLFD, the corresponding relationship between HFD (see, \cite{Kamocki}) and CFD has been obtained. In order to elaborate the significance of our results an inverse source problem involving 2nd LFD with non-local boundary conditions has been solved. We intend to investigate several useful identities in FC in the contest of $n$th LFD. 









\section*{\small
Conflict of interest} 

{\small
The authors declare that they have no conflict of interest.}




\bigskip  

\small 
\noindent
{\bf Publisher's Note}
Springer Nature remains neutral with regard to jurisdictional claims in published maps and institutional affiliations.

\end{document}